\documentclass[12pt]{amsart}
\usepackage[OT4]{fontenc}
\usepackage{graphics,graphicx}
\usepackage{amsmath}
\usepackage{amssymb}
\usepackage {amsthm}

\usepackage{xcolor}

\usepackage{enumerate}
\usepackage{hyperref}
\usepackage[cp1251]{inputenc}
\usepackage[all]{xy}
\usepackage{cleveref}
\usepackage{lipsum}

\usepackage[active]{srcltx}

\newtheorem{theorem}{Theorem}

\newtheorem{lemma}{Lemma}

\newtheorem*{proposition*}{Proposition}
\newtheorem{definition}{Definition}

\newtheorem{remark}{Remark}

\makeatletter
\renewcommand{\subsection}{\@startsection{subsection}{2}{0mm}{-\baselineskip}{-5pt}{\it \bf}}
\makeatother

\title{On primary decompositions of unital locally matrix algebras}
\author{Oksana Bezushchak and Bogdana Oliynyk }

\thanks{The second author was partially supported by the grant for scientific researchers of the ``Povir u sebe'' Ukranian Foundation.}

\begin{document}
	
	\maketitle
	
	\address{Faculty of Mechanics and Mathematics,
		Taras Shevchenko National University of Kyiv,
		Volodymyrska, 60, Kyiv 01033, Ukraine \\ Department of Mathematics, National University
		of Kyiv-Mohyla Academy, Skovorody St. 2, Kyiv,
		04070, Ukraine                     }
	
	\email{bezusch@univ.kiev.ua, oliynyk@ukma.edu.ua}

	\keywords{Keyword: locally matrix algebra, primary decomposition, Steinitz number, tensor product, Clifford algebra}               
	
	\subjclass{2010}{Mathematics Subject Classification: 03C05, 03C60, 11E88}
	


\begin{abstract}
		We construct a unital locally matrix algebra of uncountable dimension that
		\begin{enumerate}
            \item[(1)] does not admit a primary decomposition,
			\item[(2)] has an infinite  locally finite Steinitz number.
		\end{enumerate}
	
	It gives negative answers to questions from \cite{BezOl}  and  \cite{Kurochkin}. We also show that for an arbitrary infinite Steinitz number $s$ there exists a unital locally matrix algebra  $A$ having the Steinitz number $s$ and not isomorphic to a tensor product of finite dimensional matrix algebras.	
		
	\end{abstract}

	\section{Introduction}

 Let $F$ be a ground  field. In this  paper we consider associative unital $F$--algebras. Following A.~G.~Kurosh \cite{Kurosh} we say that an algebra $A$ with a unit $1_A$ is a locally matrix algebra if an arbitrary finite collection of elements $a_1,$ $\ldots,$ $a_s \in A$  lies in a subalgebra $B$,  $1_A\in B \subset A$,  that is  isomorphic to a matrix algebra $M_n(F),$ $n\geq 1.$

 In \cite{Koethe}  G.~Koethe proved that every countable dimensional unital locally matrix algebra admits a decomposition into an (infinite) tensor product of finite dimensional matrix algebras and
  admits a primary decomposition.

  A.~G.~Kurosh \cite{Kurosh} and V.~M.~Kurochkin \cite{Kurochkin} further studied  existence and  uniqueness of  such decompositions of unital locally matrix algebras of arbitrary dimensions. In particular  V.~M.~Kurochkin \cite{Kurochkin}  formulated the  question:

 \begin{center}\textit{does every locally matrix algebra have a  primary decomposition?}\end{center}

 J.~G.~Glimm \cite{Glimm} proved that a countable dimensional unital locally matrix algebra is uniquely determined by its  Steinitz  number.
In \cite{BezOl} we showed that this is no longer true for an algebra of the dimension $> \aleph_0. $ As follows from \cite{Sushch2}, \cite{Glimm}   for any Steinitz  number $s$ there exists a countable dimensional unital locally matrix algebra $A$ such that $\mathbf{st}(A)=s$. In \cite{BezOl} we proved that for any not locally finite Steinitz number $s$ there exists an uncountable dimensional unital locally matrix algebra $A$ with the Steinitz number $s$. So, in  \cite{BezOl} we asked:

 \begin{center}\textit{if  $\mathbf{st}(A)$ is locally finite, does it imply   that $\dim_F A \leq  \aleph_0 $?}
 \end{center}
	
In this paper we give negative answers to both of these questions. In  the section \ref{S3} we prove that
\begin{theorem}
	\label{teor1}
For an arbitrary infinite locally finite Steinitz  number $s$ there exists  a unital locally matrix algebra of uncountable dimension with  Steinitz  number $s$.
\end{theorem}
Moreover, in this case the algebra $A$ has no  primary decomposition.

\begin{theorem}
	\label{teor2}
There exists  a unital locally matrix algebra of uncountable dimension that has no  primary decomposition.
\end{theorem}

A.G.~Kurosh in \cite{Kurosh} also constructed an example of a unital locally matrix algebra of uncountable dimension that does not admit  a decomposition into an infinite tensor product of finite dimensional  matrix algebras. Another example of this kind (a Clifford algebra) is constructed in \cite{BezOl}. Both examples in \cite{Kurosh} and \cite{BezOl} have Steinitz number $2^\infty$. In  the section \ref{S4} for an arbitrary odd number $l$ we construct a unital locally matrix algebra of Steinitz number $l^{\infty}$ that does not admit a  decomposition into a tensor product of finite dimensional matrix algebras. Let $\mathbb{R}$ be the set of all real numbers with natural order and let $Clg(l,\mathbb{R})$ be the generalized Clifford algebra (see section \ref{S4}, a more general construction of this kind appeared in \cite{Ramakr}).
	\begin{theorem}
	\label{theor3}
If $l$ is an odd number then  $Clg(l,\mathbb{R})$ is not isomorphic to a tensor product of finite dimensional  matrix algebras.	
\end{theorem}

Finally we obtain
\begin{theorem}
	\label{theor4}
For an arbitrary infinite Steinitz number $s$ there exists  a unital locally matrix algebra $A$ such that $\mathbf{st}(A)=s$ and $A$ does not admit a  decomposition into a tensor product of finite dimensional matrix algebras.
\end{theorem}

\section{Steinitz numbers}


Let $ \mathbb{P} $ be the set of all primes and $ \mathbb{N} $ be the set of all positive integers. A
  {\it Steinitz } or {\it supernatural} number (see \cite{ST})  is an infinite formal
product of the form
\begin{equation}\label{st1}
 \prod_{p\in \mathbb{P}} p^{r_p}  , 
 \end{equation}
where   $ r_p\in  \mathbb{N} \cup \{0,\infty\}$ for all $p\in \mathbb{P}.$
Two Steinitz numbers
$$ \prod_{p\in \mathbb{P}} p^{r_p} \ \text{ and } \  \prod_{p\in \mathbb{P}} p^{k_p}  $$
 can be multiplied with
$$ \prod_{p\in \mathbb{P}} p^{r_p} \cdot  \prod_{p\in \mathbb{P}} p^{k_p}= \prod_{p\in \mathbb{P}} p^{r_p+k_p}  \ ,  $$
where we assume that $k_p \in  \mathbb{N} \cup \{0,\infty\}$ and   $t+\infty=\infty+t=\infty+\infty=\infty$ for all positive integers $t$.

Denote by $ \mathbb{SN} $ the set of all Steinitz numbers. Note, that the set of all positive integers $ \mathbb{N}$ is subset of $\mathbb{SN}$.   A  Steinitz number \eqref{st1} is called {\it locally finite } if  $r_p \neq \infty $ for any $p \in  \mathbb{P} $. The numbers  $ \mathbb{SN}\setminus   \mathbb{N}  $  are called {\it infinite}  Steinitz numbers.

Let $A$ be a  locally matrix algebra with a unit $1_A$ over a field $F$ and let $D(A)$ be  the  set of all positive integers $n$ such that  there is a subalgebra  $A'$, $1_A \in A'\subseteq A$, $A' \cong  M_n(F)$.
Then the  least common multiple of the set $D(A)$ is called the {\it Steinitz  number }  of the algebra $A$ and denoted as $\mathbf{st}(A)$.

\section{Primary decompositions of unital locally matrix algebras}
\label{S3}

\begin{definition}
	\label{primary}
A unital locally matrix algebra $A$ over a field $F$ is called primary if $\mathbf{st}(A)=p^s,$ where $p$ is a prime number and $ s \in  \mathbb{N} $ or $s =\infty $.
\end{definition}

Recall, that if $A$ and $B$ are unital locally matrix algebras then the algebra $A \otimes_F B$ is a unital  locally matrix and
$ \mathbf{st}(A \otimes_F B)=\mathbf{st}(A) \cdot \mathbf{st}(B)$ (see \cite{BezOl}).

\begin{definition}
	\label{primary decom}
We say that the decomposition $$A = \bigotimes_{p\in \mathbb{P}} A_p$$ of a  unital locally matrix algebra $A$ over $F$ is a  primary decomposition if    each algebra $A_p$ is primary for all $p\in \mathbb{P}.$
\end{definition}

\begin{proof}[Proof of Theorem \ref{teor1}] The crucial  role in the proof will be played by the theorem of A.~G.~Kurosh (\cite{Kurosh}, Theorem 10) which is reformulated as follows:
\vskip 5mm	
	
 \textit{let $A$ be a countable dimensional  locally matrix algebra with a unit $1_A$. Then $A$
contains a proper subalgebra $1_A \in B \subset A$ such that $A\cong B$.}
\vskip 5mm

 Now suppose that $A$ is a locally matrix algebra such that $\mathbf{st}(A)=s$ and all unital locally matrix algebras of Steinitz number $s$  are no more than countable dimensional.

 Let $\gamma$  be an uncountable ordinal. For all ordinals $\alpha \leq \gamma$  we will construct an unital locally matrix algebra  $A_{\alpha}$  such that
  \begin{enumerate}
   \item[(1)] $\mathbf{st}(A_{\alpha})=s,$
   \item[(2)] if $\alpha < \beta \leq \gamma$ then  $A_{\alpha}$ is properly contained in $A_{\beta}$.
    \end{enumerate}
 Let  $A_1 = A.$ If $\alpha$ is a limit  ordinal then we let $$A_{\alpha}= \bigcup_{\mu < \alpha} A_{\mu},$$ where $\mathbf{st}(A_{\mu})=s,$ so  $\mathbf{st}(A_{\alpha})=s.$

  If $\alpha$ is a nonlimit  ordinal then  $\alpha-1$ exists and an algebra $A_{\alpha-1}$  has been constructed with $\mathbf{st}(A_{\alpha-1})=s.$ By an assumption   $\dim_F A_{\alpha-1}\leq \aleph_0.$ Hence  by  Kurosh's theorem there exists an unital locally matrix algebra $A'$ such that   $A_{\alpha-1}$ is properly contained in $A'$,  the unit of  $A_{\alpha-1}$ is the unit of $A'$ and $\mathbf{st}(A')=s.$  Let $A_{\alpha}=A'.$

  We have already arrived at  contradiction since the algebra $A_{\gamma}$ can not be countable dimensional.
\end{proof}

\begin{proof}[Proof of Theorem \ref{teor2}] Let $s$ be an infinite locally finite Steinitz number. We have shown (Theorem \ref{teor1}) the existence of a unital locally matrix algebra  $A$ such that $\mathbf{st}(A)=s$ and $\dim_F A > \aleph_0.$  If the algebra $A$ admits  a primary decomposition then $$A\cong \bigotimes_{p\in \mathbb{P}} A_p, \ \ \mathbf{st}(A_p)=p^{r_p} \ \text{ and } r_p< \infty \ \text{ for all } p\in\mathbb{P}. $$ Hence $A_p \cong M_{p^{r_p}}(F)$ and therefore $\dim_F A \leq \aleph_0.$ This concludes the proof of  Theorem \ref{teor2}.
\end{proof}

\section{Decompositions into products of matrix algebras}
\label{S4}

Let's recall the definition of a generalized Clifford algebra introduced in \cite{BezOl}.

Let $l>1$ be an integer. If $\text{char }F>0$ then we assume that $l$ is coprime with  $\text{char }F$. Let $\xi \in F$ be an $l$-th primitive root of $1$. Let $I$ be an ordered set. The generalized Clifford algebra $Clg(l,I)$ is presented by generators $x_i$, $i \in I$, and relations: $$ x_i^l =1\, , \ \ x_i^{-1}x_j x_i =\xi x_j \ \text{ for }  i<j \, , $$ $$\ x_i^{-1}x_j x_i =\xi^{-1} x_j \ \text{ for }  i>j \, , \ i,j \in I \, .$$
	
	In \cite{BezOl} we showed that $Clg(l,I)$ is a unital locally matrix algebra of Steinitz number $l^\infty$ and that ordered monomials  $$x_{i_1}^{k_1}\cdots x_{i_r}^{k_r}, \ i_1< \ldots < i_r, \ 1 \leq k_j\leq l-1 , \ 1\leq j \leq r,$$ form a basis of $Clg(l,I)$.

Let $X_{\mathbb{Q}}$ be the set of generators indexed by rational numbers.

\begin{lemma}
	\label{lemma1}
	Let $l$ be odd. Then the centralizer of $X_{\mathbb{Q}}$ in  $Clg(l,\mathbb{R})$ is $F \cdot 1$.
\end{lemma}

\begin{remark}
	The assertion of the Lemma is not true for $l=2$, that is for the case of ordinary Clifford algebras. Indeed, if $\alpha$, $\beta$ are two distinct irrational numbers then $x_{\alpha}x_{\beta}$ lies in the centralizer of $X_{\mathbb{Q}}$.
\end{remark}

\begin{proof}[Proof of Lemma \ref{lemma1}]
In \cite{BezOl} we showed that
\begin{enumerate}
	\item[(1)]
	for any $i \in I$ the mapping $\varphi_i (x_k)=\xi ^{\delta_{ik}} x_k,$ $k \in I,$ extends to an automorphism $\varphi_i$ of $Clg(l,I)$,
    \item[(2)]  any subspace $V$ of $Clg(l,I)$ that is invariant under all automorphisms $\varphi_i,$ $i \in I,$ is spanned by all ordered monomials lying in $V.$
\end{enumerate}	

The centralizer of $X_{\mathbb{Q}}$ is invariant under all automorphisms $\varphi_i,$ $i \in \mathbb{R},$ hence it is spanned by all ordered monomials. Let a monomial $$v=x_{i_1}^{k_1} \cdots x_{i_r}^{k_r}, \ i_1 < \ldots < i_r, \ 1\leq k_j \leq l-1, \ 1 \leq j \leq r,$$  lies in the centralizer of $X_{\mathbb{Q}}.$

Let $j$ be a rational number such that $j< i_1.$  Then $$x_j^{-1} v x_j = \xi^{(k_1+ \cdots + k_r)} v=v,$$ which implies $k_1+ \cdots + k_r=0 \mod l. $ Now let $j$ be a rational number such that $i_1 < j < i_2.$ Then $$x_j^{-1} v x_j = \xi^{(-k_1+ k_2+ \cdots + k_r)} v=v,$$ which implies  $-k_1+ k_2+ \cdots + k_r=0 \mod l. $

Subtracting these comparisons we get $2 k_1 =0 \mod l.$  Since the number $l$ is odd we conclude that $k_1 =0  \mod l ,$ a contradiction. This completes the proof of the Lemma \ref{lemma1}.
\end{proof}

The next statement directly follows from A.~G.~Kurosh \cite{Kurosh}.

\begin{lemma}
	\label{lemma2}
 Let $A$ be a locally matrix algebra of uncountable dimension. Suppose that $A$ contains a countable subset $S$ whose centralizer is $F \cdot 1$. Then $A$ is not isomorphic to a tensor product of finite dimensional matrix algebras.
\end{lemma}
\begin{proof}
Let $A = \otimes_{i\in I} A_i,$ where $\dim_F A_i < \infty.$ Then $\text{Card } I > \aleph_0.$ The subset $S$ lies in $ \otimes_{j\in J} A_j,$ where $J$ is a countable subset of $I.$ Hence for any   $i \in I\backslash J$ the subalgebra  $ A_{i} $ lies in the centralizer of $S,$ a contradiction. This completes the proof of the Lemma \ref{lemma2}.
\end{proof}

\begin{proof}[Proof of Theorem \ref{theor3}]
From Lemma \ref{lemma1} it follows that  the centralizer of the countable subset $X_{\mathbb{Q}}$ in  $Clg(l,\mathbb{R})$ is $F \cdot 1$. Then  by Lemma \ref{lemma2} the algebra $Clg(l,\mathbb{R})$ does not admit a decomposition into a tensor product of  finite dimensional matrix algebras.  	
\end{proof}

\begin{proof}[Proof of Theorem \ref{theor4}]	\label{p1} (1) There are two examples of an  uncountable dimensional unital locally matrix algebra $A(2)$ with Steinitz number $s=2^\infty$ that is not isomorphic to a tensor product of finite dimensional matrix algebras. The first  example  was constructed by A.G. Kurosh \cite{Kurosh}. In \cite{BezOl} we found such an example among Clifford algebras. For more details see \cite{BezOl}, \cite{Kurosh}.

(2)	\label{p2} For   an odd number $l$ the  generalized Clifford algebra $A(l)=Clg(l,\mathbb{R})$ with   Steinitz  number $s=l^\infty$ also does not admit such a decomposition (Theorem \ref{theor3}).

(3) 	\label{p3}	Now let  $s$ be an infinite locally finite Steinitz number and let $A$ be the algebra from Theorem \ref{teor1} and Theorem \ref{teor2}, $\mathbf{st}(A)=s.$ If $A$ admitted a decomposition into a tensor product of finite dimensional matrix algebras then  $A$ would admit a primary decomposition.

(4) \label{p4}	Let $s$ be an infinite Steinitz number  that is not locally finite, i.e. $s=p^{\infty} \cdot s'$ for some  prime number $p$.  By  \cite{Sushch2} there exists a countable (or finite) dimensional unital locally matrix algebra $A'$  such that $\mathbf{st}(A')=s'.$ Let $A= A(p) \otimes_{F}  A'.$  Clearly, $\mathbf{st}(A)=s.$ By Lemma \ref{lemma1} and \cite{BezOl}, \cite{Kurosh} (in the case $p=2$) the algebra $ A(p)$ contains a countable dimensional subspace $W$ whose centralizer is $F \cdot 1_{A(p)}.$ The subspace $W\otimes_{F} A'$ of the algebra $A$ is also countable dimensional.

We claim that the centralizer of  $W\otimes_{F} A'$  is  $F \cdot 1_{A}$. Indeed, let an element $a=\sum_i a_i\otimes a_i'$ lie in the centralizer of $W\otimes_{F} A';$ $a_i \in A(p),$  $a_i' \in A'.$ Suppose that the elements $a_i' $ are linearly independent. For an arbitrary element $w\in W$ we have $$[\sum_i a_i\otimes a_i',w\otimes 1_{A'}]=\sum_i [a_i,w]\otimes a_i'=0,$$ which implies that the  elements $a_i$ lie in the centralizer of $W,$ that is, in  $F \cdot 1_{A(p)}.$ Hence $a= 1_{A(p)} \otimes a',$ where the element $a'$ lies in the center of $A',$ $a' \in F \cdot 1_{A'}.$ This completes the proof of the claim.

By Lemma \ref{lemma2} the algebra $A$ is not isomorphic to a tensor product of finite dimensional matrix algebras. This   completes the proof of the Theorem~\ref{theor4}.	
\end{proof}

\end{document}